\documentclass[11pt]{amsart}
\usepackage{amssymb,amsmath,mathtools,bbm,tikz,color}
\usepackage{float}
\usepackage{verbatim}
\usepackage{url}
\usepackage[letterpaper, bookmarks=true,%
    colorlinks=true,%
    linkcolor=blue,%
    citecolor=blue,%
    urlcolor=blue,%
    breaklinks=true]{hyperref}
\usepackage{combelow}
\usepackage{bbm}
\usepackage{graphicx} 
\usepackage[english]{babel} 
\usepackage{float}
\usepackage{mathtools}
\usepackage{hyperref}
\usepackage[noadjust]{cite}
\usepackage[margin=3.2cm]{geometry}
\hypersetup{ 
    colorlinks=true,       
    linkcolor=blue,          
    citecolor=blue,        
    filecolor=blue,      
    urlcolor=blue           
}

\allowdisplaybreaks

\def\C{\mathbb{C}}

\def\Z{\mathbb{Z}}
\def\R{\mathbb{R}}
\def\Q{\mathbb{Q}}
\def\be{\begin{equation}}
\def\ee{\end{equation}}
\def\ot{\otimes}
\def\t{\times}
\def\sb{\subset}
\def\a{\alpha}

\def\d{\delta}

\def\S{\Sigma}
\def\De{\Delta}
\def\ov{\overline}
\def\wt{\widetilde}
\def\id{\mathrm{id}}
\def\fsl{\mathfrak{sl}}

\def\U{U_q(\fsl(2|1))}
\def\UQ{U_{\mathbb{Q}}}
\def\HH{\mathcal{H}}

\def\CC{\mathcal{C}}
\def\TS{T_{\Sigma}}
\def\FF{\overline{F}}
\def\Vna{V(n,\a)}
\def\Bna{\mathcal{B}_{n,\a}}
\def\rev{\overrightarrow{\mathrm{ev}}}
\def\lev{\overleftarrow{\mathrm{ev}}}
\def\lcoev{\overleftarrow{\mathrm{coev}}}
\def\coev{\mathrm{coev}}

\usepackage{array}
  
\usepackage[shortlabels]{enumitem}
\usepackage{pstricks} 
\usepackage{pst-all}
\usepackage{epsfig}
\usepackage{pst-grad} 
\usepackage{pst-plot} 
\usepackage[space]{grffile} 
\usepackage{etoolbox} 
\usepackage{color} 
\makeatletter 
\patchcmd\Gread@eps{\@inputcheck#1 }{\@inputcheck"#1"\relax}{}{}
\makeatother

\usepackage[all]{xy}
\usepackage{mathrsfs}
\usepackage{graphics}
\usepackage{amsthm}

\theoremstyle{plain}\newtheorem{theorem}{Theorem}[section]\newtheorem{Theorem}{Theorem}\newtheorem{proposition}[theorem]{Proposition}\newtheorem{lemma}[theorem]{Lemma}

\theoremstyle{definition}\newtheorem{remark}[theorem]{Remark}

\def\Q{\mathbb{Q}}\def\C{\mathbb{C}}\def\Z{\mathbb{Z}}\def\R{\mathbb{R}}\def\ZZ2{\mathbb{\Z/ 2\Z}}
\def\sb{\subset}\def\ot{\otimes}\def\t{\times}\def\sm{\setminus}
\def\a{\alpha}\def\d{\delta}\def\e{\epsilon}\def\De{\Delta}\def\La{\Lambda}\def\la{\lambda}\def\la{\lambda}\def\p{\partial}\def\S{\Sigma}
\def\ov{\overline}\def\wt{\widetilde}

\def\sl2{\mathfrak{sl}_2}
\def\su2{\mathfrak{su}(2)}
\def\tr{\text{tr}\,}
\def\id{\text{id}}\def\End{\text{End}}

\def\deg{\text{deg}}

\def\Uq{U_q(\mf{sl}_2)}\def\Usl2{U_q(\sl2)}\def\usl2{\wt{U}_q(\sl2)}\def\Uq{U_q(\mf{sl}_2)}\def\uq{\mathfrak{u}_q(\sl2)}

\def\CC{\mathcal{C}}

\def\ZZ{\mathcal{Z}}

\def\CC{\mathcal{C}}

\def\HH{H=\{\Ha\}_{\a\in G}}

\def\Z{\mathbb{Z}}\def\gg{\mathfrak{g}}
\def\C{\mathbb{C}}
\def\ot{\otimes}\def\id{\text{id}}

\def\sl{\mathfrak{sl}}

\def\CC{\mathcal{C}}

\def\rev{\overrightarrow{\text{ev}}}

\def\rcoev{\overrightarrow{\text{coev}}}
\def\lev{\overleftarrow{\text{ev}}}

\def\lcoev{\overleftarrow{\text{coev}}}

\def\sl{\mathfrak{sl}}

\def\sll{\mathfrak{sl}(2|1)}
\def\uq{\mathfrak{u}_q(\sl)}
\def\HH{\mathcal{H}}

\def\top{\mathrm{top}_{\text{LG}}}

\def\HH{\mathcal{H}}

\def\Va{V_{\a}}
\def\rev{\overrightarrow{\text{ev}}}

\def\rcoev{\overrightarrow{\text{coev}}}
\def\lev{\overleftarrow{\text{ev}}}

\def\lcoev{\overleftarrow{\text{coev}}}

\def\coev{\text{coev}}

\def\XX{\mathbb{X}}

\def\CCgr{\CC_{\gg,r}}

\def\Imax{I_{max}}
\def\CCgr{\CC_{\gg,r}}

\def\TS{T_{\Sigma}}

\def\Imax{I_{max}}

\def\uqg{\mathfrak{u}_q(\gg)}

\def\Uqsl{U_q(\sl(2|1))}

\def\ka{\kappa}
\def\psl{\mathfrak{psl}}

\title[The top coefficient of Links-Gould is Murasugi-sum multiplicative]{The top coefficient of the Links-Gould invariant is Murasugi-sum multiplicative}

\author[Cristina Anghel]{Cristina Anghel}
\address{Laboratoire de math\'ematiques Blaise Pascal, Universit\'e Clermont-Auvergne, Campus des C\'ezeaux 3, place Vasarely, 63178 Aubi\`ere, France; Institute of Mathematics ``Simion Stoilow'' of the Romanian Academy, 21 Calea Grivi\cb{t}ei Street, 010702 Bucharest, Romania.}
\email{cristina.anghel@uca.fr; cranghel@imar.ro}

\author[D. López Neumann]{Daniel López Neumann}
\address{Universidad de Concepción, Concepción, Chile}
\email{dlopezn@udec.cl}

\author[R. van der Veen]{Roland van der Veen}
\address{University of Groningen, Netherlands}
\email{r.i.van.der.veen@rug.nl}

\begin{document}

\maketitle

\begin{abstract}

We prove that the Laurent polynomial in $\Z[q^{\pm 1}]$ that is the top coefficient of the Links-Gould invariant of the boundary of a Seifert surface is multiplicative under arbitrary Murasugi sums of minimal genus Seifert surfaces. This extends a previous result of the second and third authors which only holds for plumbing, i.e. 4-Murasugi sum. Our technique relies on quantum supergroup themselves and PBW basis rather than the skein relations used in the previous work.


\end{abstract}

\section{Introduction}

The Links-Gould invariant $LG(L;q,p)$ is an isotopy invariant of links $L\sb S^3$ that lives in $\Z[q^{\pm 1},p^{\pm 1}]$ \cite{LG:two-variable}. This is defined as a Reshetikhin-Turaev invariant \cite{RT1} through the representation theory of the quantized Lie superalgebra of $\sl(2|1)$, which is a ribbon tensor category \cite{Yamane:quantized-superalgebras-R-matrices, Khoroshkin-Tolstoi:universal-R-matrix-superalgebras}. The variable $q$ in $LG(L;q,p)$ is the usual quantum group variable, however, the second variable $p=q^{2\a+1}$ comes from a unique super-Lie phenomenon: it is related to ``non-discreteness'' in the category of representations, namely, the existence of a family of finite-dimensional simple modules $V_{\a}$ with non-integral highest weight $\a\in \C\sm\Z$. This is a phenomenon that happens precisely for the class of ``type I'' Lie superalgebras in Kac's classification (except for $\psl(n|n)$) \cite{Kac:Lie-superalgebras} and marks a sharp contrast with the representation theory of simple complex Lie algebras. In particular, this second variable has a completely different origin than that of the HOMFLY and Kauffman polynomials. 
\medskip

In a series of recent works \cite{KT:Links-Gould, LNV:plumbing-multiplicative, GLN:genus-for-Lie-superalgebras} it has been shown that the variable $p=q^{2\a+1}$, not only has a different origin than that of HOMFLY or Kauffman, but it also has a much different (and better) geometric behaviour: the $p$-degree (or breath) is bounded above by four times the Seifert genus \cite{KT:Links-Gould, GLN:genus-for-Lie-superalgebras}, the coefficient of the top $p$-power of a fibred link is a power of $q$ determined by the Hopf invariant and, more generally, the same top coefficient is multiplicative with respect to plumbing of (minimal genus) Seifert surfaces \cite{LNV:plumbing-multiplicative}. 
\medskip

In this paper, we generalize the main result of \cite{LNV:plumbing-multiplicative} from ``plumbing'' to arbitrary Murasugi-sum. Recall that Murasugi-sum is a way to glue two embedded surfaces in 3-space along an embedded polygon on each surface (see Fig. \ref{fig:argument-1}, left). A deep theorem of Gabai states that a Murasugi-sum has minimal genus if and only if each surface in the sum does \cite{Gabai:Murasugi1, Gabai:Murasugi2}. Here we prove:

\begin{Theorem}
\label{Theorem: main theorem for Links-Gould}
    If a link $L$ is obtained from links $L_1$ and $L_2$ by Murasugi-sum along genus-minimizing Seifert surfaces, then $$\top(L,q)=\top(L_1,q)\cdot\top(L_2,q).$$
\end{Theorem}


We now explain the meaning of $\top(L,q)$. As shown in \cite{KT:Links-Gould} (and slightly generalized to disconnected $\S$ in \cite{LNV:plumbing-multiplicative}), the Links-Gould invariant satisfies the bound 
\begin{equation}
\label{eq: bound for LG}
    \deg_pLG(L;q,p)\leq 2(1-\chi(\S))
\end{equation}
where $\S$ is any compact, oriented surface with $\p\S=L$ and no closed components. A surface satisfying these conditions and minimizing $-\chi(\S)$ is said to be genus-minimizing. Then we define $$\top(L,q) := \text{ coefficient of } p^{1-\chi(\S)} \text{ in } LG(L;q,p)$$
for any genus-minimizing Seifert surface $\S$ of $L$.
For knots, this is simply the coefficient of $p^{2g(K)}$ where $g(K)$ is the Seifert genus. Note that this coefficient is zero if the bound (\ref{eq: bound for LG}) is not sharp, for instance, for the Conway knot.
\medskip

\def\AAm{\mathcal{A}_M}
\def\AAp{\mathcal{A}_{pl}}

As a corollary of our theorem, we can (in principle) improve Corollary 2 and 3 of \cite{LNV:plumbing-multiplicative}. For each non-zero $n\in\Z$, let $A_n$ be an unknotted annulus with $n$ full twists (where $n>0$ means left-handed twists). 
Let $\AAm$ (resp. $\AAp$) be the class of links that bound surfaces obtained by iteratively Murasugi-summing or desumming $A_n$'s with $n\neq 0$ (resp. plumbing/deplumbing). \cite[Corollary 2]{LNV:plumbing-multiplicative} gives an obstruction for a link to belong to the class $\AAp$, and \cite[Corollary 3]{LNV:plumbing-multiplicative} establishes ``plumbing-uniqueness'' for links in this class. By our main theorem, both results hold if $\AAp$ is replaced by $\AAm$. However, we do not know whether $\AAp\neq \AAm$.

\medskip

An analogous Murasugi-sum multiplicativity property has been shown by the second and third authors in \cite{LNV:Murasugi-sum} for the (top coefficients of) non-semisimple link polynomials defined through unrolled quantum groups at roots of unity \cite{CGP:non-semisimple} (as long as Seifert surfaces are connected). The top coefficients considered in \cite{LNV:Murasugi-sum} are cyclotomic integers, while here they are Laurent polynomials in $q$. In this sense, we think quantum supergroups give neater results. 
Now, the root-of-unity quantum group case has the advantage that small quantum groups are at hand, and these are finite-dimensional Hopf algebras (often ribbon), hence they have cointegrals and universal invariants of bottom tangles are defined (see Remark \ref{remark: cutting lemma for root of 1 case is easier}). This is why the result of \cite{LNV:Murasugi-sum} holds for all Lie types while the present one is for a single quantum supergroup. 

\medskip
It is interesting to note that the Murasugi-sum statement does not hold for the higher Links-Gould invariants $LG^{(n)}$, that is, the analogous invariants but using the $U_q(\sl(2|1))$-modules $V(n,\a)$ with $n\geq 1$ instead of $V(0,\a)$. This can be seen from simple computations, e.g. Hopf links have $q$-monic $LG^{(1)}$, but that's not the case for the trefoil and figure eight (see e.g. \cite{Garoufalidis-Kashaev:multivariable} for the $V_2$-polynomial which conjecturally coincides with $LG^{(1)}$). 
\medskip

We now give a few comments concerning our theorem and available data on Links-Gould. Indeed, our results explain the many coincidences seen in tables of computations of Links-Gould invariants, for instance \cite{de-Wit:automatic-evaluation-LG} or the more recent and powerful \cite{Garoufalidis-Li:patterns-of-V2}. Many prime knots share the same top coefficients of LG and often this is because such knots are made by Murasugi-summing the same links, but in a different way. For instance, sorting the data of \cite{Garoufalidis-Li:patterns-of-V2}, we found that there are 510 knots (resp. $360$ knots) with $\leq 12$ crossings having top coefficient a power of $q$ times $1+3q^2$ (resp. $3+q^2$). The link $H_2=\p A_2$ ($A_2$ as above) has $\top(H_2)=3q+q^3$ so these 510 knots (resp. 360) are candidates to be Murasugi-sums of the mirror image $m(H_2)$ (resp. $H_2$) with a bunch of Hopf links, but we haven't checked this. 
\medskip

It is interesting to note that the top coefficient of the HOMFLY polynomial has a similar property, though much more restrictive and less topological: it is multiplicative under ``diagrammatic'' Murasugi-sum of Seifert surfaces obtained by Seifert's algorithm, and top coefficient stands for the coefficient of $z^{g_c}$ where $g_c$ is the canonical genus \cite{MP:skein-star-product}. Small prime knots (at least for $\leq 9$ crossings) that are Murasugi-sums of smaller links are in fact diagrammatic Murasugi sums. Thus, coincidences for the top coefficient of Links-Gould of small knots (which are explained through Murasugi-sum) also hold for the top coefficient of HOMFLY. However, this is just a small-knot phenomenon, and stops being true for knots with higher crossing number (for instance, as noted in \cite{LNV:Murasugi-sum}, $11n74$ is the first fibred knot whose HOMFLY is not $v$-monic). For the 510 knots mentioned in the previous paragraph, we checked that, up to a power of $\pm v$, 472 of them have top HOMFLY coefficient $1+v^2$ and 38 of them have top coefficient $2$. 
\medskip

\subsection{Plan of the paper} Section \ref{section: quantized Lie super} contains the definition of quantum $\sl(2|1)$ and its category of modules. We start from scratch, since it is difficult to find references on Lie superalgebras and its quantized versions that are appropriate for geometric topologists. In Section \ref{section: Links-Gould invariant} we recall the definition of the Links-Gould invariant (only assuming knowledge of the Reshetikhin-Turaev operator invariants). We prove our main theorem in Section \ref{section: proof of main theorem}. On a first reading, we suggest going through Subsection \ref{subs: proof connected case} assuming Proposition \ref{prop: prop surface and coevaluations}.

\subsection{Acknowledgements} We thank Nathan Geer for email communication. The second author thanks Noah Snyder for interesting conversations at the beginning of this project.
The first author gratefully acknowledges the support of the ANR grant ANR-24-CPJ1-0026-01 at Universit\'e Clermont Auvergne - LMBP. She also acknowledges partial support from grants provided by the Ministry of Research, Innovation and Digitisation, CNCS - UEFISCDI, project numbers  PN-IV-P2-2.1-TE-2023-2040 and PN-IV-P1-PCE-2023-2001, within PNCDI IV. The second author acknowledges support from Fondecyt INI 11261601. AI was used to write down part of the preliminaries (under the prompt ``specialize \cite{GLN:genus-for-Lie-superalgebras} to $\sl(2|1)$''), which were further checked, cleaned and expanded by the authors. It was also used to organize the data of \cite{Garoufalidis-Li:patterns-of-V2}. On the math side, AI was only used to finish the argument of Lemma \ref{prop: rearranged PBW}.

\section{Quantized Lie super-algebras} 
\label{section: quantized Lie super}

\def\SVect{\text{SVect}}
\def\Uqp{U_q(\mathfrak{p}^+)}
\def\ovz{\ov{0}}
\def\ovo{\ov{1}}

\subsection{Super-vector spaces} Recall that a super-vector space is a vector space $V$ endowed with a mod 2 grading $V=V_{\ovz}\oplus V_{\ovo}$ where vectors in $V_{\ovz}$ are said to be {\em even} and vectors in $V_{\ovo}$ are {\em odd}. Morphisms between super-vector spaces are degree-preserving. The category of super-vector spaces $\SVect$ is a symmetric tensor category, where the symmetry is given by $\tau_{U,V}(u\ot v)=(-1)^{|u||v|}v\ot u$ for homogeneous $u\in U,v\in V$. Thus, one can define algebras, Lie algebras, Hopf algebras in the category of super-vector spaces; referred as superalgebras, Lie superalgebras or Hopf superalgebras. Moreover, one can define modules over such objects inside $\SVect$, this is what we consider in this paper without any further comment.

\def\fsl{\mathfrak{sl}}
\subsection{The Lie superalgebra $\fsl(2|1)$} Let $\C^{2|1}$ be a 3-dimensional complex mod 2 graded vector space, with even part $\C_{\ov{0}}=\C^2$ and odd part $\C_{\ov{1}}=\C$. We define $\sl(2|1)$ to be the subalgebra of $\End(\C^{2|1})$ consisting of endomorphisms with zero super-trace (defined as $\tr(\pi_0fj_0)-\tr(\pi_1f j_1)$ where $j_i:\C_{\ov{i}}\to \C^{2|1}, \pi_i:\C^{2|1}\to \C_{\ov{i}}$, $i=0,1$ are inclusions/projections). Here $\End(\C^{2|1})$ denotes the space of all endomorphisms, and it is mod 2 graded if we say that $f$ is (homogeneous) of degree $i$ (mod 2) if $f(\C_{\ov{j}})\sb \C_{\ov{j+i}}$ for all $j$ (mod 2). With the bracket $[f,g]=fg-(-1)^{|f||g|}gf$, $\fsl(2|1)$ becomes a super-Lie algebra.
\medskip

The universal enveloping (super)algebra of $\sl(2|1)$ can be described as the algebra with generators $H_1,H_2,E_1,E_2,F_1,F_2$ which are all even except $E_2,F_2$ which are odd, together with the relations 
\begin{align*}
    [H_i,E_j]& =a_{ij}E_j, & [H_i,F_j]&=-a_{ij}F_j & [H_i,H_j]&=0 \\
    [E_i,F_j]&=\d_{ij}H_i & [E_2,E_2]&=0 & [F_2,F_2]&=0 \\
    [E_1,[E_1,E_2]]&=0 & [F_1,[F_1,F_2]]&=0 & 
\end{align*}
Here $(a_{ij})$ is the $2\t 2$ Cartan matrix given by $a_{11}=2, a_{12}=a_{21}=-1$ and $a_{22}=0$. If we set $E_{12}=[E_1,E_2]$ and $F_{12}=[F_1,F_2]$, then these elements have degree one and $E_{12}^2=F_{12}^2=0$ in $U(\fsl(2|1))$. In particular, if $V$ is any $U(\sll)$-module, then the operators $E_2:V\to V$ and $E_{12}:V\to V$ square to zero (similarly for $F_2,F_{12}$).

\medskip
As usual, the action of the Cartan subalgebra (= diagonal matrices) determines a root-space decomposition. In this case, there are 2 simple roots $\a_1,\a_2$ and a positive non-simple root $\a_1+\a_2$. The root $\a_1$ is even while $\a_2,\a_1+\a_2$ are odd. The Weyl vector is $2\rho=2\rho_{\ov{0}}-2\rho_{\ov{1}}$ where $2\rho_{\ov{i}}$ is the sum of roots of degree $\ov{i}$. For $\fsl(2|1)$ this is $2\rho_{\ov{0}}=\a_1, 2\rho_{\ov{1}}=\a_1+2\a_2$ so that $2\rho=-2\a_2$. The root lattice $L_R$ is the lattice generated by $\a_1,\a_2$, this has a non-degenerate pairing given by $(\a_i,\a_j)=a_{ij}$, the Cartan matrix.

\subsection{Quantum $\fsl(2|1)$} Let $q$ be a variable. The quantum supergroup $U_q(\fsl(2|1))$ is the
$\C(q)$-algebra with generators $E_1,E_2,F_1,F_2,K_1^{\pm1},K_2^{\pm1}$, where $E_2,F_2$
are odd and all other generators are even, and relations
\begin{align*}
K_iE_j&=q^{a_{ij}}E_jK_i, & K_iF_j&=q^{-a_{ij}}F_jK_i, & K_1K_2&=K_2K_1, \quad
K_iK_i^{-1}=1, \\
[E_i,F_j]&=\d_{ij}\frac{K_i-K_i^{-1}}{q-q^{-1}}, & E_2^2 &= F_2^2 = 0, & &
\end{align*}
where the Cartan matrix $(a_{ij})$ is the one of $\fsl(2|1)$, together with the $q$-Serre relations
\be
E_1^2E_2 - (q+q^{-1})E_1E_2E_1 + E_2E_1^2 = 0, \qquad
F_1^2F_2 - (q+q^{-1})F_1F_2F_1 + F_2F_1^2 = 0 \,.
\ee
Here $[x,y]=xy-(-1)^{|x||y|}yx$ denotes the super-commutator. This is a Hopf superalgebra with
\begin{align*}
\De(E_i)&=E_i\ot K_i+1\ot E_i, & \De(F_i)&=K_i^{-1}\ot F_i+F_i\ot 1,& 
 \De(K_i)&=K_i\ot K_i, \\
S(E_i)&=-E_iK_i^{-1}, & S(F_i)&=-K_iF_i, & S(K_i)&=K_i^{-1} \,.
\end{align*}
We denote by $\UQ^+$ (resp. $\UQ^-$) the $\Q(q)$-subalgebra generated by $E_1,E_2$
(resp. $F_1,F_2$), by $\UQ^0$ the $\Q(q)$-subalgebra generated by
$K_1^{\pm1},K_2^{\pm1}$, and by $\UQ$ the $\Q(q)$-subalgebra generated by all of the
generators. The quantum group is graded by the root lattice if we say that $E_i$ has degree $\a_i$, $F_i$ has degree $-\a_i$ and each $K_i$ has degree zero.


\subsection{Root vectors and PBW bases} 
We define the root vectors
\be
\label{rootvectors}
\quad E_{12} := E_1E_2 - q^{-1}E_2E_1,
\qquad
F_{12} := F_1F_2 - q^{-1}F_2F_1 \,
\ee
which are both odd. 
The following relations are easily checked to hold in $U_q(\fsl(2|1))$:
\begin{align*}
E_1E_{12} &= q\, E_{12}E_1, & E_2E_{12} &= -q\, E_{12}E_2, & E_{12}^2 &= 0, \\
F_1F_{12} &= q\, F_{12}F_1, & F_2F_{12} &= -q\, F_{12}F_2, & F_{12}^2 &= 0, &  \,
\end{align*}


By \cite{Yamane:quantized-superalgebras-R-matrices}, the monomials
\be
\label{PBW}
E^J := E_1^{j_1}E_{12}^{j_3}E_2^{j_2}, \qquad 
F^J := F_1^{j_1}F_{12}^{j_3}F_2^{j_2}, \qquad j_1 \in \Z_{\geq 0}, \,\, j_2,j_3\in\{0,1\},
\ee
form $\Q(q)$-bases of $\UQ^+$ and $\UQ^-$, respectively, and the products
\be
\label{PBW basis usual order}
K_1^{a_1}K_2^{a_2}F^JE^I
\ee
(for all possible $I,J$ and $a_1,a_2\in \Z$) form a $\Q(q)$-basis of $\UQ$, this is the usual PBW basis. 
The proof of our main theorem will use the following rearrangement of the PBW basis: 

\begin{lemma}
\label{prop: rearranged PBW}
    The quantum supergroup $U_q(\sll)$ has basis the products of the form $$K_1^{a_1}K_2^{a_2}E_1^{a}F_1^b F_{12}^{\e_1}F_{2}^{\e_2}E_{12}^{\e_1'}E_{2}^{\e_2'}$$
    where $a_1,a_2\in \Z, a,b\in \Z_{\geq 0}$ and $\e_1,\e_2,\e'_1,\e'_2\in \{0,1\}$.
\end{lemma}

\begin{proof}
   Indeed, any usual PBW vector $K_1^{a_1}K_2^{a_2}F_1^{j_1}F_{12}^{j_3}F_2^{j_2}E_1^{i_1}E_{12}^{i_3}E_2^{i_2}$ is a linear combination of the reordered PBW vectors: this is because $E_1$ commutes with $F_2$, because $$F_{12}E_1=E_1F_{12}-[K_1]_qF_2+q^{-1}F_2[K_1]_q$$
   where $[K_1]_q=\frac{K_1-K_1^{-1}}{q-q^{-1}}$ as it is readily checked, and because of the commutation relation between $E_1$ and $F_1$ (and because $K_i$'s commute with PBW vectors up to a power of $q$). By the same argument, the converse also applies. Since there is the same number of vectors in each degree (where we say $|E_i|=|F_i|=1$) and since $U_q(\sll)$ is a free $U^0$-module (finitely generated in each degree), the lemma follows.
\end{proof}

\subsection{Kac modules} 
\label{sub.kac}

Fix an integer $n\geq 0$ and let $V_n$ be the $(n+1)$-dimensional simple module
over $U_q(\fsl(2))=\langle E_1,F_1,K_1^{\pm1}\rangle$ with highest weight vector $v_0$,
so that $K_1v_0=q^nv_0$ and $V_n$ has basis $\{F_1^iv_0\}_{0\leq i\leq n}$. Given
$\a\in\C$, we extend this to an action of the subalgebra $U_q(\mathfrak{p}^+)$
generated by $E_1,F_1,K_1^{\pm1},K_2^{\pm1},E_2$ by setting
$$ 
K_2v_0=q^{\a}v_0, \qquad E_2 v_0 = 0,
$$
extending $K_2$ to $V_n$ via $K_2F_1=qF_1K_2$; this forces $E_2=0$ on $V_n$ since
$[E_2,F_1]=0$. The {\em Kac module} is the induced module 
$$
\Vna=\U\ot_{U_q(\mathfrak{p}^+)}V_n \,.
$$

\begin{lemma}(\cite{Kac:reps})
\label{lem.simple}
The module $\Vna$ has dimension $4(n+1)$. It is simple if and only if $(n,\a)$ is
typical, i.e. $\a(\a+n+1)\neq 0$. In particular, $\Vna$ is simple when
$\a\in\C\setminus\Z$.
\end{lemma}

From now on we assume $(n,\a)$ typical. From the PBW basis (\ref{PBW basis usual order}) it follows that:

\begin{lemma}
\label{lemma: PBW on V(n,a)}
The vectors
\be
\label{Bna}
F_{12}^{\e_1}F_2^{\e_2}F_1^i v_0, \qquad \e_1,\e_2\in\{0,1\}, \quad 0\leq i \leq n,
\ee
form a basis $\Bna$ of $\Vna$. 
\end{lemma}


\subsection{The braiding} 

By a {\em weight module} over $\U$ we mean a finite-dimensional module on which $K_1,K_2$
act diagonalizably with eigenvalues powers of $q$ (with possibly complex exponents).
We denote by $\CC$ the category of weight modules over $\U$. The category $\CC$ is braided in the following way: for any $U,V\in\CC$, the braiding $c_{U,V}:U\ot V\to V\ot U$ is given by 
\be
\label{cuv}
c_{U,V}(u\ot v)=\tau_{U,V}\left(\HH (\Theta \cdot (u\ot v))\right)
\ee
where 
\begin{itemize}
\item
$\tau_{U,V}(x\ot y)=(-1)^{|x||y|}y\ot x$ is the transposition (in the category of super-vector spaces), 
\item 
$\HH=q^{\sum_{i,j=1}^2 a'_{ij}H_i\ot H_j}$ as an operator, where $(a'_{ij})$ is the inverse of the Cartan matrix of $\fsl(2|1)$,
\item  
and 
$$
\Theta=\sum_{J} c_J\cdot E^J\ot F^J
$$
is the quasi-R-matrix, the sum being over the PBW monomials~\eqref{PBW}. Here $c_J\in\Q(q)$ is a coefficient, for its exact formula see \cite{Yamane:quantized-superalgebras-R-matrices}.
\end{itemize}

Note that the quasi-R-matrix has $L_R$-degree zero. Note also that although $\Theta$ is an infinite sum, only finitely many terms act nonzero
on a fixed module in $\CC$. Finally, note that the action of $\Theta$ on a vector $u\ot v$ involves a sign $(-1)^{|F^J||u|}$.

\subsection{Pivotal and ribbon structure} 
\label{sub.piv}
The category $\CC$ is pivotal and ribbon in the following way. Let $g=K_2^{-2}$ (which is precisely the Weyl vector $K_{2\rho}$, where $\rho=\rho_{\ov{0}}-\rho_{\ov{1}}$), this is a pivotal element for $\U$ in the sense that it is group-like and $S^2$ is conjugation by $g$. Thus, right and left duality for $V\in \CC$ are given by $$\rev_V(v\ot v')=(-1)^{|v|}v'(K_2^{-2}v), \ \  \rcoev_V(1)=\sum_i (-1)^{|v_i|}v'_i \ot K_2^2 v_i$$
where $(v_i)$ is a basis of $V$ with dual basis $(v'_i)$ while left evaluation/coevaluation $\lev_V, \lcoev_V$ are the usual ones for (super)vector spaces. As a consequence of duality, for every $V\in \CC$ there is an isomorphism $j_V:V^{**}\to V$ characterized by
\be
\label{jVdef}
j_V^{-1}(v)(v')=(-1)^{|v|}v'(K_2^{-2}v).
\ee
The category $\CC$ is ribbon with twists given by
$$\theta_V=(\id_V\ot\rev_V)(c_{V,V}\ot \id_{V^*})(\id_V\ot \lcoev_V).$$
On $V(0,\a)$ this is
\be
\label{twist}
\theta_{V(0,\a)}= q^{-2\a(\a+1)}\,\id_{V(0,\a)} \,,
\ee
which belongs to $\Q(q,q^{\a},q^{\a^2})$.

\section{The Links-Gould invariant}
\label{section: Links-Gould invariant}

We assume the reader is familiar with the Reshetikhin--Turaev invariants of framed, 
oriented tangles from ribbon categories \cite{RT1, Turaev:BOOK1}. In this paper, the relevant tensor category $\CC$ is the category of weight modules over $U_q(\sl(2|1))$ at generic $q$. If $T$ is a framed, oriented tangle, we denote by $F(T,V)$ its Reshetikhin--Turaev
invariant in which every component is colored with a module $V$ in $\CC$. In what follows we denote $V_{\a}=V(0,\a)$.
\medskip

The Links-Gould invariant of an oriented link $L\sb S^3$ is defined as follows. Let $\wt{L}$ be $L$ with an arbitrary choice of framing (and same orientation) and let $\wt{L}_o$ be a $(1,1)$-tangle with closure $\wt{L}$. Since
$V_{\a}$ is simple, $F(\wt{L}_o,\Va)=c_{\wt{L}_o}(\a)\cdot \id_{\Va}$ for some
$c_{\wt{L}_o}(\a)\in \Q(q,q^{\a},q^{{\a}^2})$. By~\cite{GPT:modified}, $c_{\wt{L}_o}(\a)$ is independent
of where the link was cut open\footnote{More precisely, \cite{GPT:modified} proves that $d(V_{\a})c_{\wt{L}_o}(\a)$ is independent of where the link was cut open, where $V_{\a}$ is the color of the cut-open strand and $d(V_{\a})$ is the so called ``modified'' dimension. Since links are single-colored in the present paper, we do not need to multiply by $d(V_{\a})$.}, hence it is an invariant of $\wt{L}$ and
$$
q^{2\a(\a+1)w(\wt{L})}\cdot c_{\wt{L}_o}(\a) \,,
$$
is an invariant of $L$ (here $w(\wt{L})$ is the writhe of $\wt{L}$). The last expression can be shown to be a polynomial in $q$ and $p:=q^{2\a+1}$, this is the {\em Links-Gould invariant}, which we denote by $LG(L,q,p)$\footnote{Note that Lemmas \ref{lemma: algebraic bound} and \ref{lemma: coefficients of iotaX} together with Proposition \ref{prop: Seifert formula} imply that LG of any link is a Laurent polynomial in $q^{2\a}$ (instead of merely $q^{\a}$).}. In particular, if the writhe is zero (which will be the case in the proof of our main theorem) then $$F(\wt{L}_o,\Va)=LG(L;q,q^{2\a+1})\cdot\id_{\Va}.$$

\medskip

The choice of variable $p=q^{2\a+1}$ (instead of $q^{2\a}$) is to have the symmetry property $$LG(L,q,p)=LG(L,q,p^{-1})$$
which follows easily from $V(0,\a)^*\cong V(0,-\a-1)$, see e.g. \cite[Proposition 3.3]{GLN:genus-for-Lie-superalgebras}.

\section{Proof of the main theorem}
\label{section: proof of main theorem}
\def\Uq{U_q(\sll)}
\def\Va{V_{\a}}

The proof follows the lines of the argument of \cite{LNV:Murasugi-sum}, namely:
\begin{enumerate}
\item Given a link $L$, express $F(L_o,\Va)$ via a ``Seifert surface formula'' in terms of the operator invariant $F(T_{\S},X)$ of an $X$-colored bottom tangle $T=T_{\S}$ associated to a Seifert surface $\S$ of $L$, where $X=\Va\ot\Va^*$ (this part of the argument is taken from \cite{LNV:genus-unrolled}, see also \cite{GLN:genus-for-Lie-superalgebras} for the super-case). This is Proposition \ref{prop: Seifert formula}.
    \item Prove a formula for the ``top term'' of the operator invariant $F(T_{\S},X)$  (Proposition \ref{prop: prop surface and coevaluations}). This formula says that, up to ``lower degree terms, $F(T_{\S},X)$ coincides with a bunch of coevaluations times the top term of $LG(L)$. This in turn is proved in the following steps:
    
\begin{enumerate}
    \item Separate $Z_T$ into top degree part and lower degree part (acting on $X$). Here the PBW vector bringing top degree in $q^{-2\a}$ is $F_{12}F_2E_{12}E_2$ by Lemma \ref{lemma: algebraic bound}.
    \item By Lemma \ref{lemma: ``cointegral'' cutting lemma} this vector ``cuts'' the identity of $X$ (up to lower degree terms), bringing $F(T_{\S},X)$ into the desired form.
\end{enumerate}
\item Express a Murasugi-sum in terms of operator invariants of a bottom tangle and a ``top'' tangle, to which applying Prop. \ref{prop: prop surface and coevaluations} and \ref{prop: F(TS,X) as coev but upside-down} implies the theorem in the connected surface case.

\item We deduce the non-connected $\S$ case from the connected case using an expression for $\id_X$ in terms of horizontal twists (note that we didn't knew how to carry out this generalization in \cite{LNV:Murasugi-sum}).
\end{enumerate}

\subsection{The module $X$} In what follows we denote $V_{\a}=V(0,\a)$ and $X=V_{\a}\ot V_{\a}^*$. We always consider $\Va$ with the PBW basis of Lemma \ref{lemma: PBW on V(n,a)} (which is $F_{12}^{\e_1}F_{2}^{\e_2}v_0, \e_1,\e_2\in \{0,1\}$) and we consider $\Va^*$ with the dual basis. Then we consider $X$ with the tensor product of these two bases, and $X^*$ with the dual basis. Any tensor product of $X$ and $X^*$ is henceforth considered with the tensor product of such bases, to which we simply refer as {\em standard bases}.
\medskip

We need the following properties about the module $X$ together with its standard basis. The following is \cite[Lemma 4.3]{GLN:genus-for-Lie-superalgebras} specialized to $\sl(2|1)$.

\def\maxdeg{\text{maxdeg}}

\begin{lemma}\cite{GLN:genus-for-Lie-superalgebras}
\label{lemma: algebraic bound}
    A PBW vector $K_1^{a_1}K_2^{a_2}E_1^{a}F_1^b F_{12}^{\e_1}F_{2}^{\e_2}E_{12}^{\e_1'}E_{2}^{\e_2'}$ acts on $X$ (in the standard basis) by a matrix whose coefficients are polynomials in $q^{-2\a}$ (with $\Q(q)$-coefficients) of degree $\leq \e_1'+\e_2'$. In particular, any vector of $\Uqsl_{\Q}$ acts on $X$ by a matrix whose coefficients lie in $\Q(q)[q^{-2\a}]$ and satisfy $\maxdeg_{q^{-2\a}}\leq 2$.
\end{lemma}

\def\ovB{\overline{B}}

\begin{lemma}(Cutting lemma)
\label{lemma: ``cointegral'' cutting lemma}
    There exists a Laurent polynomial $\ka(\a)\in \Q(q)[q^{\pm 2\a}]$ of the form $$\kappa(\a)=cq^{-2\a-1}+d+cq^{2\a+1}$$ where $c,d\in \Q(q)$, such that $$F_{12}F_{2}E_{12}E_{2}|_X=\ka(\a)\cdot \lcoev_{\Va}\circ\rev_{\Va}+\ovB$$
    where $\ovB:X\to X$ is an operator whose coefficients (in the standard basis) are polynomials in $q^{-2\a}$ of degree $<2$. 
\end{lemma}

Note that since $\rev_{\Va}$ involves the action of the pivot $g=K_2^{-2}$ on $\Va$, the morphism $\ka(\a)\cdot \lcoev_{\Va}\circ\rev_{\Va}$ has coefficients in $\Q(q)[q^{-2\a}]$ (of degree 2 in $q^{-2\a}$), agreeing with the statement of Lemma \ref{lemma: algebraic bound}.

\begin{proof}
    This part was simply checked in Mathematica.
\end{proof}

\begin{remark}
\label{remark: symmetry of kappa alpha}
    For the proof of the main theorem it will be important to note that $\kappa(\a)$ has $\maxdeg_{q^{\pm \a}}=1$ and satisfies the same symmetry as the Links-Gould invariant of a link (namely, it is invariant under $p\mapsto p^{-1}$ if $p=q^{2\a+1}$.)
\end{remark}

\begin{remark}
\label{remark: cutting lemma for root of 1 case is easier}
    The analogue of Lemma \ref{lemma: ``cointegral'' cutting lemma} for quantum groups at roots of unity (for any Lie type) was proved in \cite{LNV:Murasugi-sum} where instead of $F_{12}F_{2}E_{12}E_{2}$ one has $F^{\Imax}E^{\Imax}$ (there $E^{\Imax}$ and $F^{\Imax}$ are the top positive and negative PBW vectors). In that case, the module $X$ turns out to be a module over the small quantum group $\uqg$, which is a finite-dimensional Hopf algebra hence has a (unique up to scalar) non-zero cointegral $\Lambda\in\uqg$. In \cite[Lemma 5.9]{LNV:Murasugi-sum} it is shown that $F^{\Imax}E^{\Imax}$ acts on $X$ as the cointegral $\La\in\uqg$ up to lower degree terms. Then one can apply a result of \cite{CGPT} about cointegrals which essentially follows from the defining property of the cointegral and generalities about projective modules. The same argument cannot be applied for general quantum supergroups at generic $q$ since there is no finite-dimensional Hopf algebra to play the role of the small quantum group $\uqg$. 
\end{remark}
 
 Let $\iota_X:X^*\to X$ be the isomorphism defined through the pivotal structure of $\CC$ by $$X^*=(\Va\ot\Va^*)^*\cong \Va^{**}\ot \Va^*\to \Va\ot\Va^*=X$$
where the last morphism is $j_{\Va}\ot\id_{\Va^*}$, $j_{\Va}:\Va^{**}\to \Va$ being the isomorphism coming from the pivotal structure.

\begin{lemma}\cite{GLN:genus-for-Lie-superalgebras}
\label{lemma: coefficients of iotaX}
   The isomorphism $\iota_X:X^*\to X$ has coefficients in $q^{2\a}\Q(q)$ in the standard basis of $X^*$ and $X$.
\end{lemma}
\begin{proof}
     This follows simply because the coefficients of $\iota_X$ come from the isomorphism $j_{\Va}:\Va^{**}\to \Va$ which is essentially the action of $g^{-1}=K_2^2$ on $\Va$, which has coefficients in $q^{2\a}\Q(q)$ in the standard basis.
\end{proof}

\medskip

The following analogue of \cite[Lemma 3.6]{LNV:Murasugi-sum} is straightforward to check.

\begin{lemma}
\label{lemma: coev lemma}
    We have $(\iota_X\ot (\lcoev_{\Va}\circ\rev_{\Va}))\circ\rcoev_X=\lcoev_{\Va}\ot\lcoev_{\Va}.$
\end{lemma}

\def\FC{F_{\CCgr}}
\def\TS{T_{\S}}

\def\FF{\overline{F}}

\def\JJ{\mathcal{J}_X}
\def\XX{X^{\ov{\e}}}
\subsection{Seifert surface formula}\label{subs: Seifert formula}
Let $L$ be an oriented link and let $\S$ be a Seifert surface for $L$, which for this section we assume is connected. Then $\S$ can be obtained from a single disk by attaching, say, $N$ 1-handles. After an isotopy, we can suppose $\S$ is obtained by taking an $N$-component bottom tangle $\TS$\footnote{This is almost the terminology of Habiro \cite{Habiro:bottom-tangles}, except that we do not require any condition on the placement of the endpoints of the components of $T$.}, i.e. a framed tangle in $\R^2\t[0,1]$ with all endpoints on $\R^2\t\{1\}$, and then attaching a disk on top, see Figure \ref{figure:Seifert}. Let $L_o$ be the $(1,1)$-tangle obtained by opening $L$ at $p$ as indicated in the same figure.
\begin{figure}[H]
    \centering
    \includegraphics[width=8.5cm]{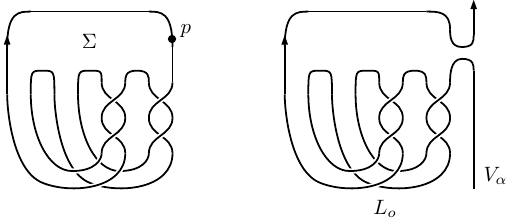}
    \caption{On the left is the Seifert surface $\S$ of $L$. On the right, we open $L=\p\S$ at $p$ to obtain a $(1,1)$-tangle $L_o$. This is colored with $\Va$ so the figure has to be interpreted as the RT invariant $F(L_o,\Va):\Va\to \Va$.}
       \label{figure:Seifert}
\end{figure}

\subsubsection{$F(L_o,\Va)$ from $F(\TS,X)$} We will express the invariant $F(L_o,\Va)$ (and hence $LG(L)$) in terms of $F(\TS,X)$ as follows. Suppose each component of $\TS$ is oriented to the right, that is, for any component $T_i$ of $\TS$ the orientation of $T_i$ is downwards at the leftmost endpoint of $T_i$ on $\R^2\t\{1\}$. Let $\e_1,\dots,\e_{2N} \in \{\pm 1\}$ be the signs at the endpoints of $\TS$ from left to right, where $\e_i=+1$ (resp. $\e_i=-1$) indicates upwards orientation  (resp. downwards). Let $X^{+1}=X$ and $X^{-1}=X^*$ and denote $\XX=X^{\e_1}\ot\dots\ot X^{\e_{2N}}$. Then the Reshetikhin-Turaev invariant of $\TS$ colored with $X$ has the form $$F(\TS,X):\C\to  \XX=X^{\e_1}\ot\dots\ot X^{\e_{2N}}.$$
Let  $\JJ:\XX\to X^{\ot 2N}$ be the isomorphism that applies $\iota_X$ to each $X^*$-tensor factor and $\id_X$ on each $X$-tensor factor. Let 
\begin{align}
    \label{eq: def of F bar}
    \FF(\TS,X)=\JJ\circ F(\TS,X):\C\to X^{\ot 2N}
\end{align}

 Then the invariant $F(L_o,\Va):\Va\to \Va$ is computed from $\FF(\TS,X)$ by
$$F(L_o,\Va)=(\id_{\Va}\ot \lev_{\Va}^{\ot 2N})\circ(\FF(\TS,X)\ot\id_{\Va}).$$
Here is a picture in the case of Figure \ref{figure:Seifert} (where $N=2$):
\begin{figure}[H]
    \centering
    \includegraphics[width=10cm]{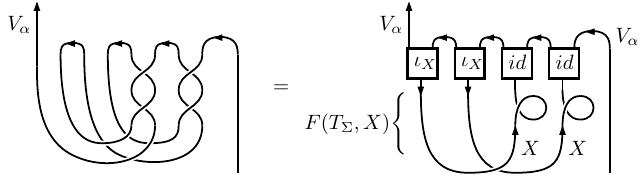}
\end{figure}
\def\ZZ{Z^{ev}}

\subsubsection{$F(\TS,X)$ from ``universal'' invariants} Let $\Uqsl_{\Q}$ be the $\Q(q)$-subalgebra of $\Uqsl$ generated by the $K_i,E_i,F_i$ for $i=1,2$. By Lemma 4.7 of \cite{GLN:genus-for-Lie-superalgebras} there is an element $$Z_{T}\in \Uqsl^{\ot N}_{\Q}$$ 
which can be thought as a ``universal'' invariant for the bottom tangle $T$. This is not exactly a universal invariant, which are usually defined for finite-dimensional ribbon Hopf algebras, but it is defined in a similar way using the fact that the braiding $c_{X,X}$ can be represented by left multiplication by an element in $U_q(\sll)_{\Q}^{\ot 2}$ (followed by a super-transposition). See \cite{GLN:genus-for-Lie-superalgebras} for more details.
\medskip

The element $Z_T$ determines $F(\TS,X)$ by left multiplication on $X^{\ot N}$ as follows: let $\ZZ_T\in \Uqsl^{\ot 2N}$ be the element obtained from $Z_{T}$ by inserting the unit $1_{\Uqsl}$ in all odd tensor factors. Let $L_{\ZZ_T}$ be the endomorphism of $(X^*\ot X)^{\ot N}$ that is left multiplication by $\ZZ_T$, so the units $1_{\Uqsl}$ in $Z_T^{ev}$ act on the $X^*$-factors and each tensor factor of $Z_T$ acts on the corresponding $X$-factor. Then
\begin{align*}
    F(\TS,X)= P_{X,X^*}\circ L_{\ZZ_T}\circ  \coev_X^{\ot N}
\end{align*}

\noindent where $\coev_X$ denotes the usual right coevaluation of super-vector spaces. Here $P_{X,X^*}$ is a permutation isomorphism of super-vector spaces (of $X$ and $X^*$ tensor factors) determined by the ordering of the components of $T$ and the way the endpoints of $T$ lie in $\R^2\t\{1\}$. Usually, there is no $R$-matrix involved in $P_{X,X^*}$ or pivot in $\coev_X$ as all these Hopf algebra elements were already included in $Z_T$. However, we prefer to have the categorical coevaluation $\rcoev_X$ in the bottom (which involves the pivot), so instead we will write

\begin{align}
\label{eq: RT of T in terms of universal}
    F(\TS,X)&= P_{X,X^*}\circ L_{(Z'_T)^{ev}}\circ  \rcoev_X^{\ot N}
\end{align}

\noindent where $Z'_T=Z_T\cdot g^{\ot N}\in \Uqsl^{\ot N}$ (recall $g=K_{2\rho}=K_2^{-2}$ for $\Uqsl$). Note that $\JJ\circ P_{X,X^*}=P_{X,X}\circ (\iota_X\ot \id_X)^{\ot N}$ where $P_{X,X}$ is the same permutation but applied to $X^{\ot N}$ so that
$\FF(\TS,X)$ can be written as $$\FF(\TS,X)=P_{X,X}\circ (\iota_X\ot \id_X)^{\ot N}\circ L_{(Z'_T)^{ev}}\circ \rcoev^{\ot N}$$
or graphically 
\begin{figure}[H]
    \centering
    \includegraphics[width=8cm]{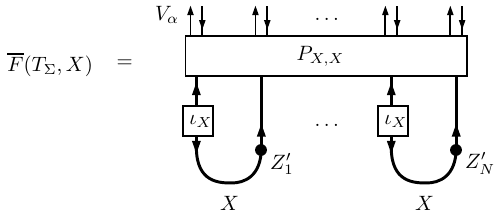}
\end{figure}
\noindent where we write $Z'_T=\sum Z'_1\ot\dots\ot Z'_N$. For the previous example:

\begin{figure}[H]
    \centering
    \includegraphics[width=10cm]{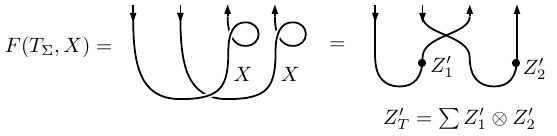}
\end{figure}

\medskip

Note that, since the quasi-R-matrix has degree zero as well as group-likes, it follows that $Z'_T$ has degree zero.
\medskip

To summarize this discussion:

\begin{proposition}
\label{prop: Seifert formula}
    If $\S$ is a connected Seifert surface of a link $L$ built from an $N$-component tangle $\TS$ as above, then the Links-Gould invariant of $L$ is determined by $\TS$ as
    $$LG(L;q,q^{2\a+1})\cdot\id_{V_{\a}}=(\id_{V_{\a}}\ot \lev_{V_{\a}}^{\ot 2N})\circ(\FF(\TS,X)\ot\id_{V_{\a}})$$
    where $\FF(\TS,X)=\JJ\circ F(\TS,X)$ is determined by left multiplication of the universal invariant $Z'_T$ acting on $X^{\ot N}$ as in (\ref{eq: RT of T in terms of universal}). 
\end{proposition}

\subsection{The main proposition}

\def\maxdeg{\text{max-deg}}

Proposition \ref{prop: prop surface and coevaluations} below is the main ingredient behind the proof of our theorem. This is the analogue of \cite[Prop. 6.2]{LNV:Murasugi-sum}. For any $Z\in \Uq^{\ot N}$ we will denote
\begin{align}
\label{eq: def of F(Z,X)}
    F(Z,X)&= P_{X,X^*}\circ L_{Z^{ev}}\circ  \rcoev_X^{\ot N}.
\end{align}
Thus $F(\TS,X)=F(Z'_T,X)$, but note that $F(Z,X)$ is not a Reshetikhin-Turaev invariant in general. We will employ similar notation for $\FF(Z,X)$.

\def\cl{\text{cl}}

\begin{proposition}
\label{prop: prop surface and coevaluations}
    Let $\TS$ be an $N$-component framed, oriented bottom tangle. Then the Reshetikhin-Turaev operator invariant $\FF(\TS,X):\C\to X^{\ot 2N}$ satisfies
\begin{align}
\label{eq prop: FTX in terms of coev + B}
     \FF(\TS,X)=c(q,q^{\a})\cdot \lcoev_{\Va}^{\ot 2N}+B
\end{align}
   \noindent for a coefficient $c(q,q^{\a})\in \Q(q)[q^{\pm 2\a}]$ and an operator $B:\C\to X^{\ot 2N}$ with the following properties:
   \begin{enumerate}[(a)]
       \item The coefficients of $B(1)\in X^{\ot 2N}$ in the standard basis lie in $\Q(q)[q^{\pm 2\a}]$ and satisfy $\maxdeg_{q^{-2\a}}<N$ and  $\maxdeg_{q^{2\a}}\leq N$. 
       \item The coefficients of the closure $$\cl(B):=(\id_{\Va}\ot \lev_{\Va}^{\ot 2N})\circ (B\ot \id_{\Va}):\Va\to\Va$$
 (which also lie in $\Q(q)[q^{\pm 2\a}]$) satisfy $\maxdeg_{q^{\pm 2\a}}<N$.
 \item $\maxdeg_{q^{\pm 2\a}}c(q,q^{\a})= N$ (except if $c(q,q^{\a})=0$) and the coefficients of $q^{\pm 2\a N}$ in $c(q,q^{\a})$ coincide with the least/top coefficient of $LG(L;q,q^{2\a+1})$. 
   \end{enumerate}
    
\end{proposition}

\begin{proof} The core is to prove $(a)$. We follow the same steps as in \cite{LNV:Murasugi-sum}.
\medskip

\textbf{Step 1: Apply algebraic bounds to $Z'_T$.} We begin by applying (\ref{eq: RT of T in terms of universal}) to the LHS of (\ref{eq prop: FTX in terms of coev + B}):
\begin{align*}
     \FF(\TS,X)=\JJ\circ F(\TS,X), \ \ F(\TS,X)=P_{X,X^*}\circ L_{(Z'_T)^{ev}}\circ  \rcoev_X^{\ot N}
\end{align*}

Applying the algebraic bound of Lemma \ref{lemma: algebraic bound} $N$ times we see that the coefficients of $L_{(Z'_T)^{ev}}\circ  \rcoev_X^{\ot N}$, and hence of $F(\TS,X)$, are polynomials in $q^{-2\a}$ with $\maxdeg_{q^{-2\a}}\leq 2 N$. Note that since there are exactly $N$ repetitions of $\iota_X$ in $\JJ$, the operator $\JJ$ has coefficients in $q^{2\a N}\Q(q)$ by Lemma \ref{lemma: coefficients of iotaX} (this has the effect of ``symmetrizing'' the powers of $q^{-2\a}$ in $F(\TS,X)$). This proves that the coefficients of $\FF(\TS,X)$ are Laurent polynomials in $q^{2\a}$ with $\maxdeg_{q^{\pm 2\a}}\leq  N$\footnote{This immediately implies the genus bound of \cite{KT:Links-Gould} and it is the $\sl(2|1)$-case of the proof of the genus bound given in \cite{GLN:genus-for-Lie-superalgebras} for all typical representations of type I Lie superalgebras.}. Our purpose then is to pin down the coefficient of $q^{-4\a N}$ in $F(\TS,X)$. The idea will be to expand $Z'_T$ in the (tensor product) PBW basis of $\Uqsl_{\Q}^{\ot N}$ and see which terms contribute (and how) to the power $q^{-4\a N}$ in $F(\TS,X)$.
\medskip

\def\uq{\mathfrak{u}_q}

\textbf{Step 2: Expand $Z'_T$ in PBW and discard low odd-E-powers.} So let's expand $Z'_T$ in the PBW basis. We assume the PBW basis of $\Uqsl_{\Q}^{\ot N}$ has been arranged as in Lemma \ref{prop: rearranged PBW}. One can write $$Z'_T=Z^{top} + Z^{low}$$
where $Z^{low}$ is the sum of all terms in the PBW-expansion of $Z'_T$ having some tensor factor of the form $K_1^{a_1}K_2^{a_2}E_1^{a}F_1^b F_{12}^{\e_1}F_{2}^{\e_2}E_{12}^{\e_1'}E_{2}^{\e_2'}$ with $\e'_1+\e'_2<2$ and $Z^{top}$ is an element of the form $$Z^{top}=\bigotimes_{i=1}^N h_i(K,F,E_1)E_{12}E_2$$
where $h_i(K,F,E_1)$ is a polynomial in $K_1^{\pm 1}, K_2^{\pm 1},E_1, F_1, F_{12},F_{2}$ with $\Q(q)$-coefficients for each $i=1,\dots,N$. Note that by Lemma \ref{lemma: algebraic bound}, $Z^{low}$ does not contribute to the terms of highest degree in $q^{-2\a}$, more precisely, $F(Z^{low},X)$ has coefficients which are polynomials in $q^{-2\a}$  with $\maxdeg_{q^{-2\a}}< 2N$. Hence
$$\FF(\TS,X)=\FF(Z'_T,X)=\FF(Z^{top},X)+\FF(Z^{low},X)$$
where all coefficients of $\FF(Z^{low},X)\in X^{\ot 2N}$ are Laurent polynomials in $q^{2\a}$ with $\maxdeg_{q^{-2\a}}< N$ and $\maxdeg_{q^{2\a}}\leq N$.
\medskip

\textbf{Step 3: The top odd-E-power.} It is thus enough to prove
\begin{align}
\label{eq proof prop: FTX in terms of coev + B'}
     \FF(Z^{top},X)=c(q,q^{\a})\cdot \lcoev_{\Va}^{\ot 2N}+B'
\end{align}

\noindent where $B'$ satisfies the same bounds as $B$ given in $(a)$. The crucial point is to note that, since $Z'_T$ has $L_R$-degree zero and $F_2^2=0=F_{12}^2$, $Z^{top}$ must have the form $$Z^{top}=\bigotimes_{i=1}^N f_i(K,E_1,F_1) F_{12}F_2E_{12}E_2$$
where each $f_i=f_i(K,E_1,F_1)$ is a polynomial in $K_1^{\pm 1}, K_2^{\pm 1}, E_1,F_1$.

\medskip

\textbf{Step 4: Applying the cutting lemma.} Now, by Lemma \ref{lemma: ``cointegral'' cutting lemma}, for each $i=1,\dots, N$ we have
\begin{align*} f_i\cdot F_{12}F_2E_{12}E_2|_{X}&=f_i\cdot (\ka(\a)\cdot\lcoev_{\Va}\circ\rev_{\Va} + \ovB)\\
    &=\e(f_i)\cdot \ka(\a)\cdot \lcoev_{\Va}\circ\rev_{\Va} + B'_i
\end{align*}
where $B'_i$ is an operator with coefficients that are polynomials in $q^{-2\a}$ satisfying $\maxdeg_{q^{-2\a}}<2$ (note that $f_i=f_i(K,E_1,F_1)$ acts by $\Q(q)$-coefficients by Lemma \ref{lemma: algebraic bound} so $B'_i:=f_i|_X\circ \ov{B}$ has the property if and only if $\ovB$ does). We used that $\lcoev_{\Va}$ is a $\Uqsl$-morphism so that $f_i|_X\circ\lcoev_{\Va}=\e(f_i)\cdot\lcoev_{\Va}$, where $\e$ is the counit of $\Uqsl$.

\medskip
Thus, writing $Z^{top}= Z^{top}_1\ot\dots\ot Z^{top}_N$ it follows that
\begin{figure}[H]
    \centering
    \includegraphics[width=14cm]{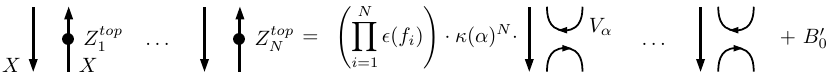}
\end{figure}

\noindent where $B'_0$ is an operator with coefficients polynomials in $q^{-2\a}$ with $\maxdeg_{q^{-2\a}}<2 N$. Thus, letting $$c(q,q^{\a})=\prod_{i=1}^N\e(f_i)\cdot \ka(\a)^N$$ and $$B'=P_{X,X}\circ (\iota_X\ot \id_X)^{\ot N}\circ B'_0\circ \rcoev_X^{\ot N},$$ we get
\begin{figure}[H]
    \centering
    \includegraphics[width=16cm]{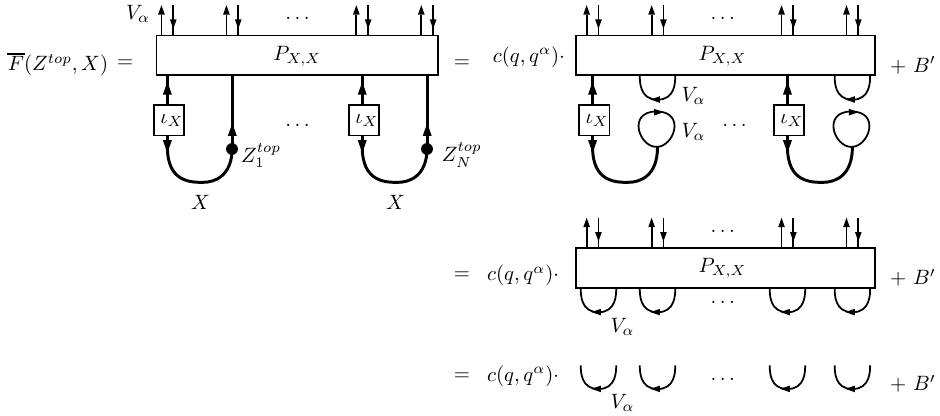}
\end{figure}

Here we used Lemma \ref{lemma: coev lemma} in the third equality while the fourth is obvious ($P_{X,X}$ is just a permutation isomorphism). Since $P_{X,X}$ and $\rcoev_X^{\ot N}$ have $\Q$-coefficients and $(\iota_X\ot \id_X)^{\ot N}$ has coefficients in $q^{2\a N}\Q(q)$, from what we said about $B'_0$ it follows that $B'$
has coefficients with $\maxdeg_{q^{-2\a}}<N, \maxdeg_{q^{2\a}}\leq N$. Summing up, we get
\begin{figure}[H]
    \includegraphics[width=13cm]{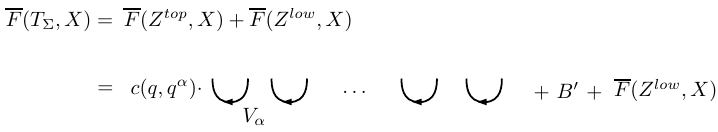}
\end{figure}
\noindent and $B:=B'+\FF(Z^{low},X)$ has coefficients with $\maxdeg_{q^{-2\a}}< N, \maxdeg_{q^{2\a}}\leq N$ since both $B'$ and $\FF(Z^{low},X)$ have this property. This proves $(a)$.
\medskip

We now prove $(c)$. Applying $\id_{\Va}\ot \lev_{\Va}^{2N}$ to (\ref{eq prop: FTX in terms of coev + B}) and using Proposition \ref{prop: Seifert formula} (on the LHS) gives $$LG(L;q,q^{2\a+1})\cdot\id_{\Va}=c(q,q^{\a})\cdot\id_{\Va}+(\id_{\Va}\ot \lev_{\Va}^{2N})\circ (B\ot \id_{\Va}).$$
Since the coefficients of $B$ satisfy $\maxdeg_{q^{-2\a}}<N$ from $(a)$ it follows that $LG(L)$ and $c(q,q^{\a})$ have the same coefficient of $q^{-2\a N}$. But since $LG(L)$ and $c(q,q^{\a})$ have the same symmetry property, namely invariance under $p\mapsto p^{-1}$ if $p=q^{2\a+1}$ (since $\ka(\a)$ does as noted in Remark \ref{remark: symmetry of kappa alpha} and $c(q,q^{\a})$ is a power of $\ka(\a)$ up to a coefficient in $\Q(q)$), it follows that $LG(L)$ and $c(q,q^{\a})$ have the same coefficient of $q^{2\a N}$ as well, therefore proving $(c)$. From this, it follows that the coefficients of $$(LG(L;q,q^{2\a+1})-c(q,q^{\a}))\cdot\id_{\Va}=\cl(B)$$
satisfy $\maxdeg_{q^{\pm 2\a}}<N$ which implies $(b)$.
\end{proof}

Let $X'=\Va^*\ot \Va$. By turning Proposition \ref{prop: prop surface and coevaluations} upside-down we also obtain: 

\begin{proposition}
\label{prop: F(TS,X) as coev but upside-down}
    Let $T_{\S'}$ be a $N'$-component (framed, oriented) top tangle. Then the operator invariant $\FF(T_{\S'},X'):(X')^{\ot 2N'}\to \C$ satisfies $$\FF(T_{\S'},X')=c'(q,q^{\a})\cdot\lev_{\Va}^{\ot 2N'}+B'$$
   \noindent for some $c'(q,q^{\a})\in \Q(q)[q^{\pm 2\a}]$ and a $B'$ with the following properties:
   \begin{enumerate}[(a)]
       \item $B'$ evaluates any standard basis vector of $(X')^{\ot 2N'}$ to a Laurent polynomial in $q^{2\a}$ with $\maxdeg_{q^{2\a}}<N'$ and $\maxdeg_{q^{-2\a}}\leq  N'$.
       \item The coefficients of $\cl(B')=(\id_{\Va}\ot B')\circ (\lcoev_{\Va}^{\ot 2N'}\ot\id_{\Va}):\Va\to \Va$
       satisfy $\maxdeg_{q^{\pm 2\a}}<N'$.
       \item $\maxdeg_{q^{\pm 2\a}}c'(q,q^{\a})= N'$ (except if $c'(q,q^{\a})=0$) and the coefficients of $q^{\pm 2\a N'}$ in $c'(q,q^{\a})$ coincide with the least/top coefficient of $LG(L';q,q^{2\a+1})$ where $L'=\p\S'$.
   \end{enumerate}

\end{proposition}

\subsection{Proof of Theorem \ref{Theorem: main theorem for Links-Gould} for connected $\S$}\label{subs: proof connected case} Now the argument is exactly the same as in \cite{LNV:Murasugi-sum}.
 Suppose $\S_1,\S_2$ are Murasugi-summed along a $2n$-gon to get the surface $\S$ and let $L=\p\S$. This $2n$-gon is topologically a disk $D$ which we can suppose lies in $\R^2\t \{0\}$ and, after sliding handles if necessary, we can suppose that both $\S_1$ and $\S_2$ are obtained by attaching 1-handles to this disk. See Figure \ref{fig:argument-1}.
 
 \begin{figure}[H]
     \includegraphics[width=11cm]{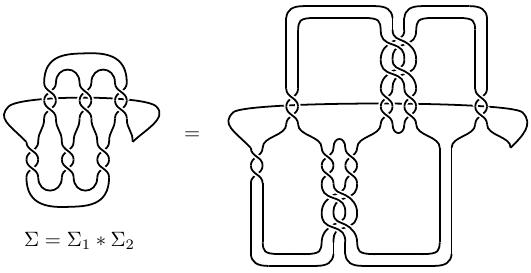}
     \caption{On the left is a Murasugi-sum of two pretzel surfaces $\S_1=\S_2=P(-2,-2,-2)$ along a hexagon. On the right, the rightmost handles of $\S_1,\S_2$ haven been slided to land on the Murasugi-sum hexagon.}
     \label{fig:argument-1}
 \end{figure}

 Suppose $\S_1$ lies below $D$ and $\S_2$ lies above, let $T_{\S_1}, T_{\S_2}$ be the corresponding tangles and let $N,N'$ be their number of components respectively. 
 Then, to compute $LG(L)$ we can open $L$ at a point of $\p D$ away from the 1-handles so that $F(L_o):\Va\to\Va$ can be written as
    \begin{align*}
        F(L_o,\Va)=(\id_{\Va}\ot \FF(T_{\S_2}, X'))\circ I\circ (\FF(T_{\S_1},X)\ot \id_{\Va})
    \end{align*}
    where $X'=\Va^*\ot\Va, X=\Va\ot\Va^*$ and the operator $I:X^{\ot 2N}\ot\Va\to\Va\ot (X')^{2N'}$ is a tensor product of left evaluations and left coevaluations (as well as identities $\id_{\Va}$ and $\id_{\Va^*}$) as specified by the (oriented) endpoints of $T_{\S_1}$ and $T_{\S_2}$. See Figure \ref{fig:argument-2}
    \begin{figure}[H]
     \includegraphics[width=11cm]{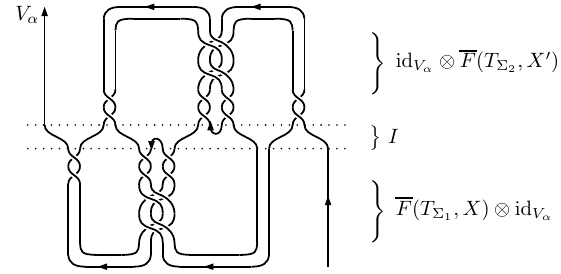}
     \caption{The link on the RHS of Figure \ref{fig:argument-1} is cut-open on the Murasugi-sum polygon. The result is a composition of morphisms as indicated. In this specific case, $I=\id_{\Va}\ot \id_{X'}\ot \lev_{\Va}\ot \id_{\Va^*}\ot \lcoev_{\Va}\ot\id_{X}\ot \id_{\Va}$.}
     \label{fig:argument-2}
 \end{figure}
 Applying Propositions \ref{prop: prop surface and coevaluations} and \ref{prop: F(TS,X) as coev but upside-down} to $\FF(T_{\S_1},X)$ and $\FF(T_{\S_2},X')$, we get:
\begin{align*}
    F(L_o,\Va)&= cc'\cdot \id_{\Va}+c'\cdot\cl(B)+c\cdot \cl(B')+(\id_{\Va}\ot B')I(B\ot \id_{\Va}) &
\end{align*}
where $c=c(q,q^{\a}),c'=c'(q,q^{\a})$. By $(a)$ and $(b)$ of Propositions \ref{prop: prop surface and coevaluations} and \ref{prop: F(TS,X) as coev but upside-down} it follows that the top term of $F(L_o,\Va)$ is the top term of $cc'\cdot \id_{\Va}$. But by $(c)$ of the same propositions, this is indeed the product of the top terms of the LG invariants of $L_1$ and $L_2$, therefore proving Theorem \ref{Theorem: main theorem for Links-Gould} whenever $\S_1,\S_2$ are connected.

\def\Sz{\S^{(0)}}
\def\So{\S^{(1)}}
\def\St{\S^{(2)}}

\subsection{Proof of Theorem \ref{Theorem: main theorem for Links-Gould} for disconnected $\S$}\label{subs: proof disconnected case} The disconnected case can be deduced from the connected case using equation (5) from \cite{LNV:plumbing-multiplicative} which implies, for any $\S$, 
\begin{align}
\label{eq: disconnected in terms of connected}
    2\top(\p\S,q)=q\top(\p\Sz,q)-(q^{-1}+q)\top(\p\So,q)+q^{-1}\top(\p\St,q)
\end{align}

\noindent where $\Sz,\So,\St$ are surfaces obtained from $\S$ by attaching a band to $\S$ with no twist, 1 twist, and 2 twists respectively. So suppose $\S$ is a Murasugi-sum of $\S_1$ and $\S_2$ and suppose that $\S_1$ is disconnected, say with $k$ components, and suppose $\S_2$ is connected. Suppose the theorem holds when Murasugi-summing a surface with $\leq k-1$ components with a connected surface. Attach a band to two different components of $\S_1$ and suppose the band is away from the Murasugi-sum. Then $\S^{(i)}=\S_1^{(i)}*\S_2$ for each $i=0,1,2$ so that
\begin{align*}
    2\top(\p\S)&=q\top(\p\Sz)-(q^{-1}+q)\top(\p\So)+q^{-1}\top(\p\St)\\
   &= q\top(\p \S_1^{(0)})\top(\p\S_2)-(q^{-1}+q)\top(\p\S_1^{(1)})\top(\p\S_2)\\
&+q^{-1}\top(\p\S_1^{(2)})\top(\p\S_2)\\
    &=2\top(\p\S_1)\top(\p\S_2)
\end{align*}
where the second equality is by the induction hypothesis and the third is by (\ref{eq: disconnected in terms of connected}) applied to $\S_1$. This completes the induction and by a similar argument, one can prove the theorem for disconnected $\S_2$.


\bibliographystyle{amsplain}
\bibliography{z-referencesabrv2}

@Article{Yamane:quantized-superalgebras-R-matrices,
 Author = {Yamane, Hiroyuki},
 Title = {Quantized enveloping algebras associated with simple {Lie} superalgebras and their universal {{\(R\)}}-matrices},
 FJournal = {Publications of the Research Institute for Mathematical Sciences, Kyoto University},
 Journal = {Publ. Res. Inst. Math. Sci.},
 ISSN = {0034-5318},
 Volume = {30},
 Number = {1},
 Pages = {15--87},
 Year = {1994},
 Language = {English},
 DOI = {10.2977/prims/1195166275},
 zbMATH = {621966},
 Zbl = {0821.17005}
}

@Article{Khoroshkin-Tolstoi:universal-R-matrix-superalgebras,
 Author = {Khoroshkin, S. M. and Tolsto{\u{\i}}, V. N.},
 Title = {Universal {{\(R\)}}-matrix for quantized (super)algebras},
 FJournal = {Communications in Mathematical Physics},
 Journal = {Commun. Math. Phys.},
 ISSN = {0010-3616},
 Volume = {141},
 Number = {3},
 Pages = {599--617},
 Year = {1991},
 Language = {English},
 DOI = {10.1007/BF02102819},
 zbMATH = {25680},
 Zbl = {0744.17015}
}

@article{Kac:Lie-superalgebras,
	Author = {Kac, V. G.},
	Doi = {10.1016/0001-8708(77)90017-2},
	Fjournal = {Advances in Mathematics},
	Issn = {0001-8708},
	Journal = {Adv. Math.},
	Language = {English},
	Pages = {8--96},
	Title = {Lie superalgebras},
	Volume = {26},
	Year = {1977},
	Zbl = {0366.17012},
	Zbmath = {3570630}}

@article{Garoufalidis-Li:patterns-of-V2,
	Author = {Stavros Garoufalidis and Shana Yunsheng Li},
	Title = {Patterns of the ${V}_2$-polynomial of knots}}

@article{Garoufalidis-Kashaev:multivariable,
	Author = {Stavros Garoufalidis and Rinat Kashaev},
	Title = {Multivariable knot polynomials from braided {H}opf algebras with automorphisms}}

@misc{Gabai:Murasugi2,
	Author = {Gabai, David},
	Howpublished = {Combinatorial methods in topology and algebraic geometry, {Proc}. {Conf}. {Hon}. {A}. {M}. {Stone}, {Rochester}/{N}.{Y}. 1982, {Contemp}. {Math}. 44, 93-100 (1985).},
	Language = {English},
	Title = {The {Murasugi} sum is a natural geometric operation. {II}},
	Year = {1985},
	Zbl = {0584.57003},
	Zbmath = {3935984}}

@misc{Gabai:Murasugi1,
	Author = {Gabai, David},
	Doi = {10.1090/conm/020/718138},
	Howpublished = {Contemp. {Math}. 20, 131-143 (1983).},
	Language = {English},
	Title = {The {Murasugi} sum is a natural geometric operation},
	Year = {1983},
	Zbl = {0524.57004},
	Zbmath = {3830877}}

@article{LG:two-variable,
	Author = {Links, J. R. and Gould, M. D.},
	Doi = {10.1007/BF00420752},
	Fjournal = {Letters in Mathematical Physics},
	Issn = {0377-9017},
	Journal = {Lett. Math. Phys.},
	Language = {English},
	Number = {3},
	Pages = {187--198},
	Title = {Two variable link polynomials from quantum supergroups},
	Volume = {26},
	Year = {1992},
	Zbl = {0777.57005},
	Zbmath = {109515}}

@Misc {KT:Links-Gould,
    AUTHOR = {Kohli, Ben-Michael and Tahar, Guillaume},
     TITLE = {A lower bound for the genus of a knot using the {L}inks-{G}ould invariant},
     year = {2023},
      NOTE = {to appear in Quantum Topol.},
}

@article{GPT:modified,
	Author = {Nathan Geer and Bertrand Patureau-Mirand and Vladimir Turaev},
	Doi = {10.1112/s0010437x08003795},
	Journal = {Compositio Mathematica},
	Month = {jan},
	Number = {1},
	Pages = {196--212},
	Publisher = {Wiley},
	Title = {Modified quantum dimensions and re-normalized link invariants},
	Url = {https://doi.org/10.1112%2Fs0010437x08003795},
	Volume = {145},
	Year = 2009}

@article{Habiro:bottom-tangles,
	Author = {Habiro, Kazuo},
	Doi = {10.2140/agt.2006.6.1113},
	Fjournal = {Algebraic \& Geometric Topology},
	Issn = {1472-2747},
	Journal = {Algebr. Geom. Topol.},
	Mrclass = {57M27 (18D10 57M25)},
	Pages = {1113--1214},
	Title = {Bottom tangles and universal invariants},
	Url = {https://mathscinet.ams.org/mathscinet-getitem?mr=2253443},
	Volume = {6},
	Year = {2006}}

@article{CGP:non-semisimple,
	Author = {Costantino, Francesco and Geer, Nathan and Patureau-Mirand, Bertrand},
	Doi = {10.1112/jtopol/jtu006},
	Fjournal = {Journal of Topology},
	Issn = {1753-8416},
	Journal = {J. Topol.},
	Mrclass = {57M27 (17B37)},
	Mrreviewer = {Leandro Vendramin},
	Number = {4},
	Pages = {1005--1053},
	Title = {Quantum invariants of 3-manifolds via link surgery presentations and non-semi-simple categories},
	Url = {https://mathscinet-ams-org.rproxy.sc.univ-paris-diderot.fr/mathscinet-getitem?mr=3286896},
	Volume = {7},
	Year = {2014}}

@article{CGPT,
	Author = {F. Costantino and N. Geer and B. Patureau-Mirand and V. Turaev},
	Journal = {arXiv:1809.07991},
	Title = {Kuperberg and {T}uraev-{V}iro Invariants in Unimodular Categories},
	Year = {2018}}

@article{RT1,
	Author = {Reshetikhin, N. and Turaev, V. G.},
	Fjournal = {Communications in Mathematical Physics},
	Issn = {0010-3616},
	Journal = {Comm. Math. Phys.},
	Mrclass = {57M25 (16W30 17B35)},
	Mrreviewer = {Louis H. Kauffman},
	Number = {1},
	Pages = {1--26},
	Title = {Ribbon graphs and their invariants derived from quantum groups},
	Url = {https://mathscinet-ams-org.rproxy.sc.univ-paris-diderot.fr/mathscinet-getitem?mr=1036112},
	Volume = {127},
	Year = {1990}}

@article{de-Wit:automatic-evaluation-LG,
 author = {De Wit, David},
 title = {Automatic evaluation of the {Links}-{Gould} invariant for all prime knots of up to 10 crossings},
 fjournal = {Journal of Knot Theory and its Ramifications},
 journal = {J. Knot Theory Ramifications},
 issn = {0218-2165},
 volume = {9},
 number = {3},
 pages = {311--339},
 year = {2000},
 language = {English},
 doi = {10.1142/S0218216500000153},
 zbMATH = {1773653},
 Zbl = {0999.57014}
}

@article{LNV:genus-unrolled,
 author = {Neumann, Daniel L{\'o}pez and van der Veen, Roland},
 title = {Genus bounds from unrolled quantum groups at roots of unity},
 fjournal = {Journal of the London Mathematical Society. Second Series},
 journal = {J. Lond. Math. Soc., II. Ser.},
 issn = {0024-6107},
 volume = {112},
 year = {2025}

}

@article{LNV:Murasugi-sum,
 Author = {Daniel L\'opez Neumann and Roland van der Veen},
 Title = {Top coefficients of {ADO} invariants are {M}urasugi-sum multiplicative},
 Year = {2026},
 Journal = {arXiv:2610.08638},
}

@article{MP:skein-star-product,
 author = {Murasugi, Kunio and Przytycki, Jozef H.},
 title = {The skein polynomial of a planar star product of two links},
 fjournal = {Mathematical Proceedings of the Cambridge Philosophical Society},
 journal = {Math. Proc. Camb. Philos. Soc.},
 issn = {0305-0041},
 volume = {106},
 number = {2},
 pages = {273--276},
 year = {1989},
 language = {English},
 doi = {10.1017/S0305004100078099},
 zbMATH = {4214956},
 Zbl = {0734.57010}
}

@inproceedings {Kac:reps,
    AUTHOR = {Kac, Victor},
     TITLE = {Representations of classical {L}ie superalgebras},
 BOOKTITLE = {Differential geometrical methods in mathematical physics, {II}
              ({P}roc. {C}onf., {U}niv. {B}onn, {B}onn, 1977)},
    SERIES = {Lecture Notes in Math.},
    VOLUME = {676},
     PAGES = {597--626},
 PUBLISHER = {Springer, Berlin},
      YEAR = {1978},
   MRCLASS = {17B70 (81E10)},
  MRNUMBER = {519631},
MRREVIEWER = {H. de Vries},
}

@article{GLN:genus-for-Lie-superalgebras,
 Author = {Garoufalidis, Stavros and {\relax L}{\'o}pez Neumann, Daniel},
 Title = {Genus bounds for knot polynomials of {L}ie superalgebras},
 Year = {2026},
 Journal = {	arXiv:2607.11735},
}

@article{LNV:plumbing-multiplicative,
 Author = {Daniel L\'opez Neumann and Roland van der Veen},
 Title = {A plumbing-multiplicative function from the {Links}-{Gould} invariant},
 Year = {2025},
 Journal = { {arXiv}:2502.12899},
 URL = {https://arxiv.org/abs/2502.12899},
 arXiv = {arXiv:2502.12899}
}

@book{Turaev:BOOK1,
	Author = {Turaev, V. G.},
	Isbn = {3-11-013704-6},
	Mrclass = {57M25 (16W30 17B37 57N10 81R50)},
	Mrreviewer = {Louis H. Kauffman},
	Pages = {x+588},
	Publisher = {Walter de Gruyter \& Co., Berlin},
	Series = {De Gruyter Studies in Mathematics},
	Title = {Quantum invariants of knots and 3-manifolds},
	Url = {https://mathscinet-ams-org.rproxy.sc.univ-paris-diderot.fr/mathscinet-getitem?mr=1292673},
	Volume = {18},
	Year = {1994}}

\end{document}